\documentclass[11pt]{article}

\usepackage{amsmath, amssymb, amsthm} 

\usepackage{dsfont}

\allowdisplaybreaks
\expandafter\let\expandafter\oldproof\csname\string\proof\endcsname
\let\oldendproof\endproof
\renewenvironment{proof}[1][\proofname]{%
	\oldproof[\bf #1]%
}{\oldendproof}

\allowdisplaybreaks

\theoremstyle{plain}
\newtheorem{theorem}{Theorem}
\newtheorem{lemma}{Lemma}[section]

\newcommand{\poly}{\text{poly}}
\RequirePackage[normalem]{ulem} 
\RequirePackage{color}\definecolor{RED}{rgb}{1,0,0}\definecolor{BLUE}{rgb}{0,0,1} 

\title{On R\"odl's Theorem for Cographs}

\author{Lior Gishboliner \thanks{Department of Mathematics, ETH, Z\"urich, Switzerland. Email: lior.gishboliner$@$math.ethz.ch. Research supported by SNSF grant 200021\_196965.} \and Asaf Shapira \thanks{School of Mathematics, Tel Aviv University, Tel Aviv 69978, Israel. Email: asafico$@$tau.ac.il. Supported in part by ERC Consolidator Grant 863438 and NSF-BSF Grant 20196.}}

\begin{document}
\date{}
\maketitle

\begin{abstract}
A theorem of R\"odl states that for every fixed $F$ and $\varepsilon>0$ there
is $\delta=\delta_F(\varepsilon)$ so that every induced $F$-free graph contains a vertex set of size $\delta n$ whose edge density is either
at most $\varepsilon$ or at least $1-\varepsilon$. R\"odl's proof relied on the regularity lemma, hence it supplied only a tower-type bound for $\delta$. Fox and Sudakov conjectured that $\delta$ can be made polynomial in $\varepsilon$, and a recent result of
Fox, Nguyen, Scott and Seymour shows that this conjecture holds when $F=P_4$. In fact, they show that the same conclusion holds even
if $G$ contains few copies of $P_4$. In this note we give a short proof of a more general statement.
\end{abstract}

\section{Introduction}\label{sec:intro}

Our investigation here is related to two of the most well studied problems in extremal graph theory.
A graph-family $\mathcal{F}$ has the {\em Erd\H{o}s-Hajnal property} if there is $c=c(\mathcal{F})>0$ such that every $n$-vertex induced $\mathcal{F}$-free graph has a clique of independent set of size at least $cn^{c}$.
The famous Erd\H{o}s-Hajnal conjecture \cite{EH} states that every non-empty family of graphs $\mathcal{F}$ has the Erd\H{o}s-Hajnal property.
A variant of the Erd\H{o}s-Hajnal conjecture was obtained by R\"odl \cite{Rodl}, who proved that if $G$ is induced $F$-free then for every $\varepsilon > 0$, $G$ contains a set of vertices of size $\delta_F(\varepsilon)\cdot n$ whose edge density is either
at most $\varepsilon$ or at least $1-\varepsilon$ (we will henceforth call such sets $\varepsilon$-homogenous).
R\"odl's proof relied on Szemer\'edi's regularity lemma, and thus supplied very weak tower-type bounds for $\delta_F(\varepsilon)$.
Fox and Sudakov \cite{FS} obtained
a quantitative improvement over R\"odl's proof by showing that one can take $\delta_F(\varepsilon)=\varepsilon^{O_F(\log 1/\varepsilon)}$.
They further conjectured that the statement holds already when $\delta=\varepsilon^{O_F(1)}$, noting that a solution of their conjecture would also resolve the Erd\H{o}s-Hajnal conjecture.

Motivated by Nikiforov's \cite{Nikiforov} strengthening of R\"odl's theorem, Fox, Nguyen, Scott and Seymour \cite{FNSS} introduced the following variant of the conjecture raised in \cite{FS}.
Let us say that a graph $F$ on $f$ vertices is {\em viral} if for every $\varepsilon>0$, there is $\delta=\varepsilon^{O_F(1)}$
so that every graph $G$ that contains at most $\delta n^f$ induced
copies of $F$ must contain an $\varepsilon$-homogenous set of size $\delta n$. The main result of \cite{FNSS} was that $P_4$, the path on $4$ vertices, is viral. Our aim in this paper is to give a very short proof of this result. In fact, we prove the following much more general statement.

\begin{theorem}\label{thm:main}
Let $\mathcal{F}$ be a finite graph family which contains a bipartite, a co-bipartite and a split graph. Suppose that $\mathcal{F}$ satisfies the Erd\H{o}s-Hajnal property. Then for every $\varepsilon > 0$ there is $\delta = \poly(\varepsilon) > 0$, where the polynomial depends on $\mathcal{F}$, such that the following holds. For every $n$-vertex graph $G$, $n \geq 1/\delta$, if $G$ contains less than $\delta n^{|V(F)|}$ induced copies of $F$ for every $F \in \mathcal{F}$, then $G$ contains an $\varepsilon$-homogenous set $X$ of size $|X| \geq \delta n$.
\end{theorem}

\noindent

Since $P_4$ is bipartite, co-bipartite and split, and since it is well known that the Erd\H{o}s-Hajnal conjecture holds for $P_4$,
the fact that $P_4$ is viral follows immediately from Theorem \ref{thm:main}.
As another example, it is well known that $G$ is a line graphs if and only if it is induced ${\cal F}$-free, where ${\cal F}$
is a family of $9$ graphs. Since ${\cal F}$ satisfies the assertion of Theorem \ref{thm:main}, we deduce that if $G$ contains
only $\poly(\varepsilon)n^{|V(F)|}$ induced copies of each of the $9$ graphs $F \in {\cal F}$, then $G$ contains an $\varepsilon$-homogenous set of size $\poly(\varepsilon) \cdot n$.
We finally mention that we can prove a version of Theorem \ref{thm:main} which applies to some infinite families ${\cal F}$ is analogy with
Theorem 5 in \cite{GS}; the details are omitted.

One of the tools used in \cite{FNSS} is the polynomial removal lemma for $P_4$ of \cite{AF}.
While our proof of Theorem \ref{thm:main} is inspired by an alternative proof of this result
in \cite{GS}, our proof here is much simpler. It is in fact very similar to the Regularity+Tur\'an+Ramsey
proof technique that was first introduced in \cite{Rodl} and later used in numerous works applying the regularity method.
The key differences which give the improved polynomial bound are that we replace the application
of Szemer\'edi's regularity lemma, with an application of the Alon-Fischer-Newman regularity lemma \cite{AFN}, and that we replace
the application of Ramsey's theorem, with an application of the (assumed) Erd\H{o}s-Hajnal property.

\section{Proof of Theorem \ref{thm:main}}\label{sec:proof}

We will need the following lemma.

\begin{lemma}\label{lem:AFN}
Let $\mathcal{F}$ be a family of graphs which contains a bipartite, a co-bipartite and a split graph. Then for every $\gamma > 0$ there is $\beta = \poly(\gamma) > 0$ (where the polynomial depends on $\mathcal{F}$) such that the following holds. For every $k_0 \geq 1$ and every $n$-vertex graph $G$, $n \geq k_0/\beta$, if $G$ contains less than $\beta n^{|V(F)|}$ induced copies of $F$ for every $F \in \mathcal{F}$, then $G$ has an equipartition into $k$ parts $V_1,\dots,V_k$, where $k_0 \leq k \leq k_0/\beta$, such that for all but $\gamma \binom{k}{2}$ of the pairs $1 \leq i < j \leq k$ it holds that $d(V_i,V_j) \geq 1-\gamma$ or $d(V_i,V_j) \leq \gamma$.
\end{lemma}

Lemma \ref{lem:AFN} follows immediately from the conditional regularity lemma of Alon, Fischer and Newman \cite{AFN} for graphs with bounded VC-dimension (see also \cite{LS}), and the fact that if $\mathcal{F}$ contains a bipartite graph, a co-bipartite graph and a split graph, then the family of induced $\mathcal{F}$-free graphs has bounded VC-dimension (see Lemma 2.2 in \cite{GS}).

\begin{proof}[Proof of Theorem \ref{thm:main}]
		Let $f := \max_{F \in \mathcal{F}}|V(F)|$ and $c := c(\mathcal{F})$.
		Let $\varepsilon > 0$. The required $\delta(\varepsilon) = \poly(\varepsilon)$ will be given implicitly by the proof.
		Set
		$$
		\gamma := \min\left\{\frac{1}{f^2}, \; \frac{\varepsilon}{4}, \; \frac{1}{2}\left(\frac{c\varepsilon}{4}\right)^{1/c}\right\}$$
		and $k_0 := \lceil 1/\gamma \rceil$. Let $\beta = \beta(\gamma)$ be given by Lemma \ref{lem:AFN}. Note that $\beta = \poly(\varepsilon)$ because $\gamma = \poly(\varepsilon)$.
		Let $G$ be an $n$-vertex graph. We may assume that $G$ contains less than $\beta n^{|V(F)|}$ induced copies of $F$ for every $F \in \mathcal{F}$. Let then $V_1,\dots,V_k$ be a partition as guaranteed by Lemma \ref{lem:AFN}. Note that $k = \poly(1/\varepsilon)$ because $\beta = \poly(\varepsilon)$ and $k_0 = \poly(1/\varepsilon)$.
		Define an auxiliary graph $R'$ on $[k]$ where $ij \in E(R')$ if $d(V_i,V_j) \geq 1-\gamma$ or $d(V_i,V_j) \leq \gamma$. Lemma \ref{lem:AFN} guarantees that $e(R') \geq (1-\gamma)\binom{k}{2} \geq (1 - 2\gamma)\frac{k^2}{2}$, using that $k \geq k_0 \geq 1/\gamma$. By Tur\'an's theorem, $R'$ contains a clique $A$ of size $|A| = \lceil \frac{1}{2\gamma} \rceil$. Now define an auxiliary graph $R$ on $A$ where $ij$ is an edge if $d(V_i,V_j) \geq 1-\gamma$ and $ij$ is a non-edge if $d(V_i,V_j) \leq \gamma$.
		
\paragraph{Case 1: $R$ contains an induced copy of some $F \in \mathcal{F}$.}
Let $i_1,\dots,i_m \in A$ be the vertices of such a copy; so $m = |V(F)| \leq f$. Sample $v_{i_j} \in V_{i_j}$ uniformly at random and independently, $j = 1,\dots,m$. By the definition of $R$, for each $i_ji_{\ell} \in E(F)$ we have $d(V_{i_j},V_{i_{\ell}}) \geq 1-\gamma$, and for each $i_ji_{\ell} \notin E(F)$ we have $d(V_{i_j},V_{i_{\ell}}) \leq \gamma$. By the union bound, the probability that $v_{i_1},\dots,v_{i_m}$ {\bf do not} span an induced copy of $F$ is at most $\binom{m}{2}\varepsilon \leq \binom{f}{2}\gamma < \frac{1}{2}$, using that $\gamma < 1/f^2$. It follows that $G$ has at least $\frac{1}{2}|V_{i_1}| \cdot \dots \cdot |V_{i_m}| = \frac{1}{2}(n/k)^m = \poly(\varepsilon) \cdot n^{|V(F)|}$ induced copies of $F$, completing the proof in this case.
		
\paragraph{Case 2: $R$ is induced $\mathcal{F}$-free.}
By the choice of $c = c(\mathcal{F})$, the graph $R$ contains a clique or independent set $B \subseteq A$ of size $|B| \geq c|A|^{c} \geq c(\frac{1}{2\gamma})^c \geq 4/\varepsilon$, using our choice of $\gamma$. Suppose without loss of generality that $B$ is an independent set, and write $B = \{i_1,\dots,i_t\}$. For every $1 \leq j < \ell \leq t$, we have $d(V_{i_j},V_{i_{\ell}}) \leq \gamma \leq \frac{\varepsilon}{4}$. Also, the number of edges which are contained in one of the sets $V_{i_1},\dots,V_{i_t}$ is at most $t \cdot \binom{n/k}{2} \leq \frac{t n^2}{2k^2} \leq \frac{\varepsilon t^2n^2}{8k^2}$, using that $t \geq 4/\varepsilon$. Hence, setting $X = V_{i_1} \cup \dots \cup V_{i_t}$, we have $|X| = t\frac{n}{k} = \poly(\varepsilon)n$ and
		$$
		e(X) \leq \frac{\varepsilon t^2 n^2}{8k^2} + \binom{t}{2} \frac{\varepsilon}{4} \cdot \frac{n^2}{k^2} \leq \frac{\varepsilon t^2 n^2}{4k^2}.
		$$
		As $\binom{|X|}{2} = \binom{tn/k}{2} \geq \frac{t^2n^2}{4k^2}$, we have that
		$$
		d(X) = \frac{e(X)}{\binom{X}{2}} \leq \frac{\varepsilon t^2 n^2/(4k^2)}{t^2n^2/(4k^2)} \leq \varepsilon,
		$$
		as required.
\end{proof}

\end{document}